\documentclass[12pt, a4paper,oneside]{amsart} 
\usepackage[T1]{fontenc}
\usepackage[utf8]{inputenc}
\usepackage{amssymb, amsfonts, mathrsfs}
\usepackage{url}
\usepackage{textcomp}
\usepackage[bookmarks]{hyperref}

\usepackage[maxbibnames=99]{biblatex}
\DeclareFieldFormat*{url}{}
\DeclareFieldFormat[unpublished]{url}{\mkbibacro{URL}\addcolon\space\url{#1}}
\DeclareFieldFormat*{urldate}{}
\DeclareFieldFormat[unpublished]{urldate}{\mkbibparens{\bibstring{urlseen}\space#1}}

\usepackage{lmodern}

\usepackage{amsmath}
\usepackage{amsthm}
\usepackage{mathtools}  
\mathtoolsset{showonlyrefs}  

\usepackage{tikz}

\usepackage{tkz-euclide}
\usetikzlibrary{calc}

\usetikzlibrary{patterns}
\usetikzlibrary{shapes.geometric}
\usepackage{thmtools, thm-restate}

\usepackage{lmodern} 

\usepackage[textwidth=1in,backgroundcolor=white,linecolor=white,bordercolor=white,textcolor=red]{todonotes}

\usepackage{amsmath}
\usepackage{amsthm}
\usepackage{mathtools}  
\mathtoolsset{showonlyrefs}  

\usepackage{tkz-euclide}
\usepackage{tikz}
\usetikzlibrary{patterns}
\usepackage{thmtools, thm-restate}
\newtheorem{theorem}{Theorem}
\newtheorem{definition}{Definition}
\newtheorem*{theorem*}{Theorem}
\newtheorem{lemma}[theorem]{Lemma}
\newtheorem{propo}{Proposition}[section]
\theoremstyle{definition}

\theoremstyle{remark}

\theoremstyle{remark}
\newtheorem{exam}[propo]{Example}

\newcommand\R{\mathbb{R}}
\newcommand\N{\mathbb{N}}

\newcommand\HH{\mathbb{H}}
\newcommand\Sp{\mathbb{S}}

\newcommand\TT{\mathcal{T}}

\newcommand\MM{\mathscr{M}}

\newcommand\ti[1]{\widetilde{#1}}
\newcommand\tq{\; | \;}

\newcommand\Isom{\mathrm{Isom}}

\newcommand{\cat}[1]{\mathrm{CAT}\left( #1 \right)}
\newcommand{\Eq}[1]{\mathrm{Eq}\left( #1 \right)}
\newcommand{\ps}[1]{\left< #1 \right>}

\usepackage{mathtools}
\DeclarePairedDelimiter\abs{\lvert}{\rvert}

\newcommand\Lip{\mathrm{Lip}}
\newcommand\En{\mathcal{E}}

\title[Dominating surface group representations]{Dominating $ \cat{-1} $ surface group representations
by Fuchsian ones}
\author{Florestan {Martin-Baillon}}

\begin{document}

\maketitle


\begin{abstract}
	We show
	that for every representation
	$ \rho : \pi_{1} (S_{g}) \to \Isom(X) $ 
	of the fundamental group
	of a genus $ g \ge 2 $ surface
	to the isometry group of
	a complete
	$ \cat{-1} $ metric space $ X $
	there exists
	a Fuchsian representation $ j $
	and
	a $ (j, \rho) $-equivariant
	map from $ \HH^{2} $ to $ X $
	which is $ c $ -Lipschitz
	for some $ c < 1 $,
	or $ \rho $ restricts to
	a Fuchsian representation.
\end{abstract}

\tableofcontents

\section{Introduction}

Let $ S $ be a closed surface of genus $ g \ge 2 $.
We denote its fundamental group by $ \Gamma $.
Let $ X_{1} $ and $ X_{2} $ be metric spaces.
For two representations $ \rho_{i} : \Gamma \to \Isom(X_{i}) $
$ (i=1,2) $, we say that $ \rho_{1} $ dominates $ \rho_{2} $
if there exists a $1$-Lipschitz and $ (\rho_{1}, \rho_{2}) $-equivariant
map from $ X_{1} $ to $ X_{2} $.
Such a map is called a domination.
We say that $ \rho_{1} $ dominates strictly $ \rho_{2} $ if
there exists a map which is 
$ (\rho_{1}, \rho_{2}) $-equivariant
and
$ c $-Lipschitz for
a $ c < 1 $.
We say that a representation $ j : \Gamma \to \Isom(\HH^{2}) $ is
Fuchsian if it is the holonomy of a hyperbolic structure on $ S $.

\begin{theorem}
	\label{th:main}
	Let $ X $ be a $ \cat{-1} $ complete metric space
	and $ {\rho: \Gamma \to \Isom(X)} $ a representation.
	Then there exists a Fuchsian representation $ j $
	which dominates $ \rho $.
	Moreover, either $ j $ dominates strictly $ \rho $
	or the domination is an isometric embedding.
\end{theorem}

Remark that in the second case, $ \rho $ stabilises a subset of
$ X $ which is isometric to $ \HH^{2} $ and in restriction to which
it is Fuchsian.

This theorem is known
in several special cases.
Guéritaud, Kassel and Wolff
\cite{gueritaudCompactAntideSitter2015}
proved it
for $ X = \HH^{2} $,
by showing that
every such $ \rho $
is the holonomy of
a folded hyperbolic structure.
Deroin and Tholozan
\cite{deroinDominatingSurfaceGroup2016}
proved it
for $ X $
a smooth, complete, simply connected
Riemannian manifold of
sectional curvature bounded above
by $ -1 $.
They construct a equivariant harmonic
map (with respect to an arbitrary conformal structure)
and show that one can chose
a hyperbolic structure on $ S $
that make this map 1-Lipschitz.
Daskalopoulos, Mese, Sanders and Vdovina
\cite{daskalopoulosSurfaceGroupsActing2019}
proved it for $ X $
a complete $ \cat{-1} $ metric space
and $ \rho $ convex cocompact.
Their proof relies on harmonic analysis
in singular spaces, as developed
by Koraveer-Schoen
and Mese.
They need the convex cocompact assumption
to construct harmonic conformal equivariant maps.

Our approach to tackle this
problem share an important feature
with the works of
Deroin-Tholozan
and
Daskalopoulos-Mese-Sanders-Vdovina:
we use harmonic maps
to construct domination.
However, we use \emph{discrete}
harmonic maps.
This makes the proof
less technically involved and more combinatorial,
as we don't rely on the
machinery of Sobolev
maps with metric space target.

As observed by 
Deroin-Tholozan
and
Daskalopoulos-Mese-Sanders-Vdovina,
this result gives
another proof of a
particular case of a result
of Bonk-Kleiner
\cite{bonkRigidityQuasiFuchsianActions2004}
(conjectured by Bourdon \cite{bourdonStructureConformeAu1995}):
\begin{theorem}
	[\cite{bonkRigidityQuasiFuchsianActions2004}]
	If $ \rho : \Gamma \to \Isom(X) $ 
	is a convex cocompact isometric acting
	on a $\cat{-1}$ space
	$ X $ 
	then
	the Haussdorf dimension
	of the limit set of $ \rho $
	is $ \ge 1 $,
	with equality
	if and only if
	$ \rho $
	fixes
	a geodesically embedded copy
	of $ \HH^{2} $ in $ X $
	whose induced action is Fuchsian.
\end{theorem}

The relation of domination
has been
studied in order to classify
the compact anti-de Sitter spaces
of dimension 3.
These spaces are classified
by pairs
$ (j, \rho) $
of representations
$ \Gamma \to \Isom(\HH^2) $
which act faithfully and
properly
discontinuously
on $ PSL(2, \R) $,
by simultaneous
multplication on the left
and on the right
(and $ \Gamma $
is the fundamental
group of a surface
of genus $ g \ge 2 $),
by
\cite{goldmanNonstandardLorentzSpace1985},
\cite{klinglerCompletudeVarietesLorentziennes1996},
\cite{kulkarni3dimensionalLorentzSpaceforms1985}.
Such a pair is called
\emph{admissible}.
By results of
Salein
\cite{saleinVarietesAntideSitter2000}
and Kassel
\cite{kasselQuotientsCompactsEspaces2009},
a pair
is admissible if and
only if
$ j $ is Fuchsian
and $ \rho $ is strictly
dominated by $j$.
By the theorem
\ref{th:main},
such pairs exist
for any $ \rho : \Gamma \to PSL(2, \R) $
of non maximal Euler class,
and it is proved in
\cite{gueritaudCompactAntideSitter2015}
that each Fuchsian representations $ j $
dominates strictly
a $ \rho $ whose non maximal Euler class
can be speficied.
Moreover, Tholozan
\cite{tholozanDominatingSurfaceGroup2017}
describes the space of Fuchsian
representations
which dominates a given $ \rho $:
he produces an explicit diffeomorphism
between this space and
the Teichmüller space of the surface.

The notion of domination has been
generalized to higher rank linear
representations
and 
has been proved to be a useful
concept to construct
Anosov representations,
see \cite{gueritaudAnosovRepresentationsProper2017}.


The article is organized as follows.
In section \ref{sec:notions}
we review the standard notions
of metric geometry we need.
In section \ref{sec:harm_triang},
we study metric triangulation
and show the existence
of \emph{harmonic equivariant map}.
This allow us to construct,
in section \ref{sec:con_surf},
a \emph{triangulated conical hyperbolic structure}
which dominates the representation $ \rho $.
In section \ref{sec:desing} we explain
how to handle the eventual singularites
of this conical structure.
Finaly, in section \ref{sec:rigidity}
we show that the domination is an isometric
embedding when the domination is not strict.

\textbf{Acknowledgements.}
The author wants to thanks
his advisor Bertrand Deroin for introducing him
to this subject and for his constant support.
He also thanks Nicolas Tholozan for many
useful discussions
and Anton Petrunin for his explanation
of Reshetnyak's Majorization theorem.

\section{Notions of metric geometry}
\label{sec:notions}

We introduce the notions
of metric geometry
we use in this article.
Our references
are
\cite{haefligerMetricSpacesNonpositive1999}
and
\cite{alexanderAlexandrovGeometryPreliminary2019}.

We denote the distance
between two points
$ x $ and $ y $
in a metric space
$ X $
by
$ \abs{xy} $,
$ \abs{x,y} $
or
$ \abs{xy}_{X} $
depending on the context.
A geodesic in a metric space
is an isometric embedding
of a closed interval
of the real line.
A metric space is geodesic
if every two points
are joined by a geodesic.
We work principally with
complete geodesic metric space.
In a geodesic space,
we denote by $ [xy] $
any geodesic between
$ x $ and $ y $.
A metric is $ D $-geodesic
if every two points
at distance less that $ D $
are joined by a geodesic.

Let $ \kappa \in \R $
define $ \MM_{\kappa} $
to be the
\emph{model plane of curvature $ \kappa $ },
that is the unique simply connected manifold
of dimension 2 with constant sectional
curvature $ \kappa $.
In this article we work with $ \kappa = -1 $,
$ \kappa = 1 $
and
$ \kappa = 0 $,
and in these cases
we have
$ \MM_{-1} = \HH^{2} $
the hyperbolic plane
,
$ \MM_{1} = \Sp^{2} $
the unit sphere of dimension 2
and
$ \MM_{0} = \R^{2} $
the Euclidian plane.
Denote by $ D_{\kappa} $
the diameter of $ \MM_{\kappa} $
which is infinite for $ \kappa \le 0 $
and $ \pi \sqrt{\kappa} $
for $ \kappa > 0 $.

Fix a $ \kappa \in \R $.
In a geodesic space $ X $,
consider three points
$ x,y,z $.
A triangle $ \Delta = [x,y,z] $
is a choice of geodesic
between each pair of points.
If
$ \abs{xy} + \abs{yz} + \abs{zx} < 2 D_{\kappa} $,
there exists a unique (up to isometry)
triangle in $ \MM_{\kappa} $
with the same side lengths.
Denote by $ \ti{\Delta} = [\ti{x} \ti{y} \ti{z}] $
this triangle.
We call it the comparison triangle of $ \Delta $.
Define a map $ \ti{\Delta} \to \Delta $
by sending
$ \ti{x} \mapsto x $,
$ \ti{y} \mapsto y $,
$ \ti{z} \mapsto z $,
and each side isometrically to
the corresponding side.
This map is called the comparison map
of $ \Delta $.
To a point $ p $ of $ \Delta $
corresponds a point $ \ti{p} $
of $ \ti{\Delta} $ that
we call the comparison point
of $ p $.

\begin{definition}
	[{\cite[Ch. II.1]{haefligerMetricSpacesNonpositive1999}}]
	\label{def:cat_k}
	A $ D_{\kappa} $-geodesic metric space
	$ X $ is a $ \cat{\kappa} $ space
	if for every geodesic triangle
	$ \Delta $ with perimeter $ < 2 D_{\kappa} $,
	the comparison map $ \ti{\Delta} \to \Delta $
	is
	1-Lipschitz.
	This means that for every
	$ x,y \in \Delta $,
	the comparison points
	$ \ti{x}, \ti{y} \in \ti{\Delta} $
	satisfy
	\begin{equation}
		\abs{xy}_{X}
		\le
		\abs{\ti{x} \ti{y}}
		_{\MM_{\kappa} } 
		.
	\end{equation}
	We say that the triangle
	$ \Delta $ satisfies
	the $ \cat{\kappa} $ inequality
	(see figure \ref{fig:cat_ineq}).
	A space which is locally $ \cat{\kappa} $
	is said to have \emph{curvature $ \le \kappa $}.
\end{definition}

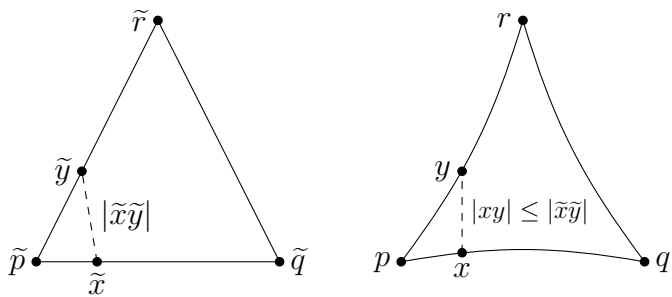
\begin{figure}[h]
	\centering
	\begin{tikzpicture}[scale=0.8]
	\coordinate (P) at (0,0);
	\coordinate (Q) at (4,0);
	\coordinate (R) at (2,4);

	\coordinate (X) at (1,0);
	\coordinate (Y) at (0.75,1.5);

	\coordinate (P') at (6,0);
	\coordinate (Q') at (4+6,0);
	\coordinate (R') at (2+6,4);

	\coordinate (X') at (1+6,0.15); 
	\coordinate (Y') at (6 + 0.75+0.25,1.5); 

	\draw (P) node[anchor=east]{$\ti{p}$}
	to (Q) node[anchor=west]{$\ti{q}$}
	to (R) node[anchor=east]{$\ti{r}$}
	to cycle;
	\filldraw [black] (P) circle (2pt);
	\filldraw [black] (Q) circle (2pt);
	\filldraw [black] (R) circle (2pt);
	\filldraw [black] (X) node[anchor=north]{$\ti{x}$} circle (2pt);
	\filldraw [black] (Y) node[anchor=east]{$\ti{y}$} circle (2pt);

	\draw[dashed] (X)
	-- node[right]{$\abs{\ti{x} \ti{y} }$}
	(Y);

	\draw (P') node[anchor=east]{$p$}
	to[bend left=10] (Q') node[anchor=west]{$q$}
	to[bend left=10] (R') node[anchor=east]{$r$}
	to[bend left=10] cycle;
	\filldraw [black] (P') circle (2pt);
	\filldraw [black] (Q') circle (2pt);
	\filldraw [black] (R') circle (2pt);
	\filldraw [black] (X') node[anchor=north]{$x$} circle (2pt);
	\filldraw [black] (Y') node[anchor=east]{$y$} circle (2pt);

	\draw[dashed] (X')
	-- node[right, scale=0.8]{$\abs{x y } \le \abs{\ti{x} \ti{y} }$}
	(Y');
\end{tikzpicture} 
	\caption{A comparison triangle for $ \kappa=0 $}
	\label{fig:cat_ineq}
\end{figure}

We fix a $ \kappa \in \R $
and $ X $ a $ \cat{\kappa} $ space.
A $ \cat{\kappa} $ space is
$ D_{\kappa} $-uniquely geodesic:
two points at distance less than $ D_\kappa $
are joined by a unique geodesic.

We define the angle between two
geodesic in a $ \cat{\kappa} $ space.
Let $ c_{1} $ and $ c_{2} $
be two geodesics starting at a
point $ x \in X $.
For a small $ t > 0 $,
consider the triangle
$ \Delta = [x c_{1} (t) c_{2} (t)] $.
Denote by
$ \ti{\angle}_{x}^{\kappa} (c_{1}(t), c_{2} (t))  $ 
the angle at $ \ti{x} $
of the comparison triangle $ \ti{\Delta} $
in $ \MM_{\kappa} $.
Then the function
$ t \mapsto 
\ti{\angle}_{x}^{\kappa} (c_{1}(t), c_{2} (t)) $
is non-increasing
and its limit when
$ t $ goes to 0
is denoted
$\angle (c_{1}, c_{2}) $
and is called
the angle between $ c_{1} $
and $ c_{2} $.
We use the notation
$ \angle_{x} (y,z) =
\angle ([xy], [xz]) $.
The angle function satisfies
a triangle inequality
\begin{equation}
	\angle( c_{1}, c_{3} )
	\le
	\angle( c_{1}, c_{2} )
	+
	\angle( c_{2}, c_{3} )
	,
\end{equation}
for every geodesics
$ c_{1}, c_{2}, c_{3} $
starting at the same point.
This means that
the angle function
induces a pseudo-metric
on the set of
geodesic starting at a point $ x $.
Two geodesics with angle 0
define the same direction.
The associated metric space
is the \emph{space of direction
of $ X $ at $ x $ }
and is denoted $ \Sigma_{x} X $.
Element of this
space are directions
and the distance between
two such directions
is given by the angle
between two representatives.
The completion of
$ \Sigma_{x} X $
is a $ \cat{1} $ space
\cite[Th. II.3.19]{haefligerMetricSpacesNonpositive1999} 
which we confund with
$ \Sigma_{x} X $.

A subset $ C $ of $ X $
is convex if for each $ x,y \in C$,
the geodesic
$ [xy] $
lies entirely in $ C $.
Suppose in this paragraph
that $ \kappa \le 0 $.
A closed convex subset
of a $ \cat{\kappa} $
space is itself
a $ \cat{\kappa} $ space.
If $ C $ is a closed subset
of $ X $
there exists a retraction
$ \pi : X \to C $
called the
\emph{projection on the convex $ C $}
which maps a point $ x $
to the closed point $ \pi (x) \in C $
\cite[Prop. II.2.19]{haefligerMetricSpacesNonpositive1999}.
This projection is a 1-Lipschitz map.

Le $ X_{1} $ and $ X_{2} $ be
two $ \cat{\kappa} $ spaces
with convex subspaces
$ C_{1} \subset X_{1} $
and
$ C_{2} \subset X_{2} $.
If $ i : C_{1} \to C_{2} $
is an isometry,
by \emph{Reshetnyak's gluing theorem}
\cite[Th. II.11.1]{haefligerMetricSpacesNonpositive1999},
the gluing of $ X_{1} $
and $ X_{2} $ along $ i $
is a $ \cat{\kappa} $ space.

Consider a triangle
$ \Delta = [xyz] $
with perimeter $ < 2 D_{\kappa} $
in $ X $.
By the $ \cat{\kappa} $ inequality
we know that for each pair
of points $ p,q \in \Delta $
we have
\begin{equation}
	\abs{pq}
	\le
	\abs{\ti{p} \ti{q}}
	,
\end{equation}
where $ \ti{p} $
and $ \ti{q} $
are the comparison
points in the comparison
triangle
$ \ti{\Delta} $.
If for $ p,q $
on different sides
of $ \Delta $
this inequality
is an equality:
$ 
\abs{pq}
=
\abs{\ti{p} \ti{q}}
$,
we say we are
in \emph{the rigidity
case of the
$ \cat{\kappa} $
inequality
}.
In this case,
the comparison
map
$ \ti{\Delta} \to \Delta $
extends to an isometry
between
the convex hull
of $ \ti{\Delta} $
in $ \MM_{\kappa} $
and the convex hull
of $ \Delta $ in $ X $
\cite[Prop. II.2.9]{haefligerMetricSpacesNonpositive1999}.
If one of the angle of $ \Delta $
is equal to the corresponding
comparison angle in $ \ti{\Delta} $
we are also in the rigidity case
of the $ \cat{\kappa} $ inequality
and the same conclusion holds.

The $ \cat{\kappa} $ 
inequality for triangles
can be extended
to general curves:
\begin{theorem}
	[{Reshetnyak's majorization theorem
	\cite[8.12.4]{alexanderAlexandrovGeometryPreliminary2019}}]
	\label{th:majorization}
	A closed curve $ \sigma $
	of length $ < 2 D_{\kappa} $
	in a $ \cat{\kappa} $ space $ X $
	is majorized by a convex
	$ C \subset \MM_{\kappa} $.
	This means that
	there exists
	a convex $ C \subset \MM_{\kappa} $
	and a 1-Lipschitz map
	$ M:C \to X $
	such that the restriction
	of $ M $ to
	$ \partial C $ is
	a length parametrization
	of $ \sigma $.
\end{theorem}

An useful case of this theorem
is when the curve $ \sigma $
is a geodesic closed polygon
$ [x_1 \cdots x_{n}] $,
that is the curve
obtained by concatenating
the
geodesic segments
$ [x_{i} x_{i+1}] $.
In this case the convex
$ C $ can be choosen
such that its boundary
$ \partial C $
is a geodesic polygon
with same side lengths
and the map $ M $
sends vertices to
corresponding vertices,
see \cite[8.12.14]{alexanderAlexandrovGeometryPreliminary2019}
and also the discussion of section
\ref{sec:rigidity}.

For a triangle $ \Delta $
in $ X $,
the comparison map
$ \ti{\Delta} \to \Delta $
is defined on the sides
of the triangle $ \ti{\Delta} $.
The majorization theorem \ref{th:majorization}
allows us to extend this map
to the convex hull of $ \ti{\Delta} $.

\section{Harmonic triangulation}
\label{sec:harm_triang}

The first part of the proof is to find
a conical hyperbolic structure which
dominates $ \rho $.
To construct this structure, we start with
a topological triangulation of the surface,
and then upgrade it to a hyperbolic simplicial
complex.
For that we construct a discrete harmonic map.

\subsection{Triangulation, metric triangulation}
\label{sub:triangulation}

We fix a triangulation $ \TT $ of the surface $ S $,
where $ \TT = (V, E, F) $ and $ V $ is the set of vertices,
$ E $ the set of edges and $ F $ the set of faces.
We can choose, for example,
the Riemann triangulation:
in the standard $ 4g $-gon folding
model of $ S $, connect one
vertex to all of the other
(see figure \ref{fig:triangulation}).
We lift it to a triangulation of $ \ti{S} $
that we denote
$ \ti{\TT} = (\ti{V}, \ti{E}, \ti{F}) $.

\begin{figure}[h]
	\centering
		\begin{tikzpicture}[scale=0.8]
		\node[draw,minimum size=4cm,name=O,regular polygon,regular
		polygon sides=8] {};
		\draw
		(O.corner 1) node[fill, circle, inner sep=1pt]{}
		-- node[above]{$a_{1}$}
		(O.corner 2) node[fill, circle, inner sep=1pt]{};
		\draw
		(O.corner 2) node[fill, circle, inner sep=1pt]{}
		-- node[above left]{$b_{1}$}
		(O.corner 3) node[fill, circle, inner sep=1pt]{};
		\draw
		(O.corner 3) node[fill, circle, inner sep=1pt]{}
		-- node[left]{$a_{1}^{-1} $}
		(O.corner 4) node[fill, circle, inner sep=1pt]{};
		\draw
		(O.corner 4) node[fill, circle, inner sep=1pt]{}
		-- node[below left]{$b_{1}^{-1} $}
		(O.corner 5) node[fill, circle, inner sep=1pt]{};
		\draw
		(O.corner 5) node[fill, circle, inner sep=1pt]{}
		-- node[below]{$a_{2} $}
		(O.corner 6) node[fill, circle, inner sep=1pt]{};
		\draw
		(O.corner 6) node[fill, circle, inner sep=1pt]{}
		-- node[below right]{$b_{2} $}
		(O.corner 7) node[fill, circle, inner sep=1pt]{};
		\draw
		(O.corner 7) node[fill, circle, inner sep=1pt]{}
		-- node[right]{$a_{2}^{-1} $}
		(O.corner 8) node[fill, circle, inner sep=1pt]{};
		\draw
		(O.corner 8) node[fill, circle, inner sep=1pt]{}
		-- node[above right]{$b_{2}^{-1} $}
		(O.corner 1) node[fill, circle, inner sep=1pt]{};

		\draw
		(O.corner 1) node[fill, circle, inner sep=1pt]{}
		-- 
		(O.corner 3) node[fill, circle, inner sep=1pt]{};
		\draw
		(O.corner 1) node[fill, circle, inner sep=1pt]{}
		-- 
		(O.corner 4) node[fill, circle, inner sep=1pt]{};
		\draw
		(O.corner 1) node[fill, circle, inner sep=1pt]{}
		-- 
		(O.corner 5) node[fill, circle, inner sep=1pt]{};
		\draw
		(O.corner 1) node[fill, circle, inner sep=1pt]{}
		-- 
		(O.corner 6) node[fill, circle, inner sep=1pt]{};
		\draw
		(O.corner 1) node[fill, circle, inner sep=1pt]{}
		-- 
		(O.corner 7) node[fill, circle, inner sep=1pt]{};

	\end{tikzpicture} 
	\caption{The Riemann triangulation for a genus 2 surface.}
	\label{fig:triangulation}
\end{figure}
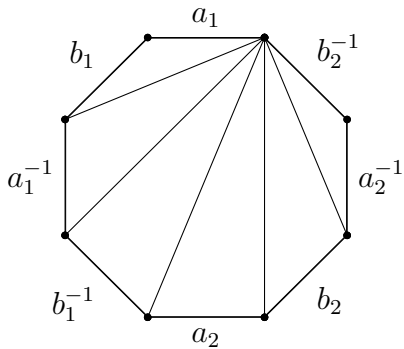

A length function on a triangulation
$ \TT $  is a function
$ \ell: E \to \R_{\ge 0} $ which satisfies the triangle
inequality: for all faces $ f $,
if $ e_{1}, e_{2}, e_{3} $ are the edges adjacents
to $ f $ then we have
\begin{equation}
	\ell (e_{1})
	\le
	\ell (e_{2})
	+
	\ell (e_{3})
	.
\end{equation}
If for a face $ f $, the inequality
above is an equality, we say
that $ \ell $ flatten the face $ f $.
It happens in particular if one
of the edges $ e $ of $ f $
is given length 0: in this case
we say that $ \ell $
flatten the edge $ e $.

The data
$ (\TT, \ell) $
of a triangulation with a length function 
is called a metric triangulation.
If $ \ell $ is a length function on $ \ti{\TT} $ which is
invariant by deck transformations, then it induces a length
function on $ \TT $ that we denote by the same symbol.

A map $ F : \ti{V} \to X $ which
is equivariant with respect to a representation
$ \rho : \Gamma \to \Isom(X) $ defines a length
function $ \ell_{F} $ on $ \ti{\TT} $ by the formula
$ \ell_{F} (e) = \abs{F(e_{-}) F(e_{+})} $,
where $ e_{-} $ and $ e_{+} $ are the vertices of $ e $.
By equivariance of $ F $,
$ \ell_{F} $ is invariant by deck transformation
and induces a length function on $ \TT $.

\subsection{Energy, harmonic map}
\label{sub:energy_harmonic_map}

We define the energy of
a map $ F: \ti{V} \to X $ which is
$ \rho $-equivariant.
Let $ \TT_{0} = (V_{0}, E_{0}, F_{0}) $
be a subtriangulation of $ \ti{\TT} $
which is a fundamental domain for the
action by deck transformations.
The energy of $ F $ is the following quantity:
\begin{equation}
	\En(F)
	=
	\sum
	_{e \in E_{0} }
	\ell_{F} (e)^{2}
	=
	\sum
	_{e \in E_{0} }
	\abs{F(e_{-}) F(e_{+})} 
	^{2}
	.
\end{equation}
We say that such an equivariant map
$ F: \ti{E} \to X $ is harmonic
if it minimizes the energy among all such maps.

We first show that harmonic maps exist.
\begin{theorem}
	\label{th:harm_exist}
	If $ \rho: \Gamma \to \Isom(X) $
	is a representation which
	does not fix a point on the boundary
	at infinity of $ X $
	then there exists a harmonic $ \rho $-equivariant
	map.
\end{theorem}
To prove existence, we use results from
\cite{jostNonpositiveCurvatureGeometric1997a} and
\cite{capraceInfinityFinitedimensionalCAT2010}
that we state now.
Let $ (N,d) $ be a complete $ \cat{0} $ metric space.
Let $ \Phi : N \to \R $ be a convex, lower semicontinuous
function which is bounded from below.
For $ \lambda > 0 $ we define the
Moreau-Yosida approximation of $ \Phi $ by:
\begin{equation}
	\Phi^{\lambda} 
	=
	\inf_{y \in N}
	\left(
		\lambda \Phi (y)
		+
		\abs{py}^{2} 
	\right)
	,
\end{equation}
where $ p \in N $ is some fixed base point.
For every $ \lambda >0 $ there exists
a unique $ y_{\lambda} \in N $ such that
$ \Phi^{\lambda} = \Phi (y_{\lambda}) $
\cite[Lemma 3.1.2]{jostNonpositiveCurvatureGeometric1997a}.
\begin{theorem}[{\cite[Theorem 3.1.1]{jostNonpositiveCurvatureGeometric1997a}}]
	\label{jost_min}
	For such a function $ \Phi $,
	if for some sequence $ (\lambda_{n}) $
	of real numbers
	going to infinity
	the sequence $ (y_{\lambda_{n} } ) $ is bounded,
	then $ y_{\lambda} $ converges to a minimizer
	of $ \Phi $ when $ \lambda $ goes to infinity.
\end{theorem}

Following \cite{capraceInfinityFinitedimensionalCAT2010},
we say that the action of a group $ G $ on
a $ \cat{0} $ space is evanescent if there
exists an unbounded set $ T $ such that
for all compact $ K \subset G $,
the set
$ \left\{ \abs{x, gx} ; x \in T, g \in K \right\} $
is bounded.
\begin{theorem}[{\cite[Proposition 1.8]{capraceInfinityFinitedimensionalCAT2010}}]
	\label{evanescent_fix}
	Let $ G $ act on a $ \cat{-1} $ complete metric space $ X $.
	If the action is evanescent then
	$ G $ fixes a point in $ X \cup \partial X $.
\end{theorem}

We can now prove the existence
of harmonic map.
\begin{proof}[Proof of Theorem \ref{th:harm_exist}]
	We will apply Theorem \ref{jost_min} to the
	function $ \En $.
	This function is defined on
	$ \Eq{\rho} $,
	the space of $ \rho $-equivariant
	maps from $ \ti{E} $ to $ X $.
	This is
	a complete $ \cat{0} $ space with
	the product metric.
	The function $ \En $
	satisfies the regularity assumptions of the theorem.
	If the sequence
	$ (F_{n}) $ such that
	$ \En^{n} = \En ( F_{n} ) $
	is bounded in $ \Eq{\rho} $,
	then the theorem applies and $ \En $
	admits a minimizer.

	If on the contrary the sequence $ (F_{n}) $ 
	is unbounded in $ \Eq{\rho} $,
	we show that a minimizer still exists.

	For a function $ F \in \Eq{\rho} $
	denote by $ \Lip(F) $ its Lipschitz
	constant with respect to the graph
	metric on $ \ti{V} $ (that is, such
	that adjacent vertices are at distance 1).
	Observe that $ \Lip(F) \le \sqrt{\En(F)} $.

	Now, if the sequence $ (F_{n}) $ 
	is unbounded in $ \Eq{\rho} $,
	then there exists $ v \in \ti{V} $
	such that $ (F_{n}(v)) $ is unbounded
	in $ X $,
	and we have for all $ \gamma \in \Gamma $:
	\begin{equation}
		\abs{ F_{n} (v),
		\rho(\gamma)F_{n}(v)}
		\le
		\Lip(F_{n})
		\abs{v, \gamma v}
		\le
		\sqrt{\En(F_{n})}
		\abs{v, \gamma v}
		,
	\end{equation}
	and as
	$\En(F_{n})$
	is bounded,
	the function $ x \mapsto \abs{x, \rho(\gamma)x} $
	is bounded on the unbounded set
	$ \left\{ F_{n} (v); n \in \N \right\} $.
	This implies that the action of $ \Gamma $
	on $ X $ given by $ \rho $ is evanescent.
	By Theorem \ref{evanescent_fix}, 
	$ \rho $ has a fixed point in
	$ X \cup \partial X $.
	By assumption, $ \rho $ does not fix
	a point in $ \partial X $ so it
	fixes a point $ p_{0} $ in $ X $.
	In this case, the constant function
	$ v \mapsto p_{0} $ is $ \rho $-equivariant
	and is of energy $ 0 $, so is a minimizer.
\end{proof}

The harmonic map is not unique in general.
Understanding when it is is
interesting question.

\begin{exam}
	Consider a real tree $ T $
	and two hyperbolic isometries
	$ a $ and $ b $.
	Recall that the axe of an hyperbolic
	isometry $ \gamma $ of $ T $
	is the geodesic where
	$ \abs{x, \gamma x} $ 
	is minimial.
	Suppose that 
	the axes of $ a $ and $ b $
	intersect along a non-trivial
	geodesic segment.

	Remark that
	for the standard presentation
	\begin{equation}
		\ps{a_{1}, b_{1}, \dots, a_{g}, b_{g}
		\tq [a_{1}, b_{1}] \cdots [a_{g}, b_{g}]=e }
	\end{equation} 
	of $ \Gamma $,
	$ a_{1} $
	and
	$ a_{2} $
	generate a free group.
	Consider the representation
	$ \rho : \Gamma \to \Isom(T) $ 
	obtained by
	sending
	$ a_{1}  $ to $ a $,
	$ a_{2} $ to $ b $,
	and the other generators
	to the identity.

	Let $ \TT = (V, E, F) $ be
	the Riemann triangulation on $ S $
	(see figure \ref{fig:triangulation}).
	The energy of a $ \rho $-equivariant
	map $ F : \ti{V} \to T $
	is
	\begin{align}
		\mathcal{E} (F)
		& =
		\abs{
			F(x_{0} ),
			a F( x_{0}  )
		}^{2} 
		+
		\abs{
			F(x_{0} ),
			b F( x_{0}  )
		}^{2} 
		,
	\end{align}
	where $ x_{0} $ is any lift
	of the unique vertex of $ \TT $.
	The energy is minimized
	exactly when $ F(x_{0}) $
	belongs to the axes
	of $ a $ and $ b $.
	As these axes intersect
	along a segment,
	the harmonic map is not unique.
\end{exam}

\subsection{Properties of harmonic map}
\label{sub:property_harm}

Given a $ \rho $-equivariant
harmonic map $ F: \ti{V} \to X $,
we say that the metric triangulation
$ (\ti{\TT}, \ell_{F}) $ 
is a
harmonic triangulation.
In this section we prove that
the conical angles of an harmonic triangulation
are all greater than $ 2 \pi $
and we construct the singular
hyperbolic surface associated
to a harmonic triangulation.

First we prove the following
result:
\begin{propo}
	\label{prop:con_angles}
	Let $ F : \ti{V} \to X $
	be an harmonic,
	$ \rho $-equivariant map
	which does not flatten
	any edge,
	and
	$ x \in \ti{V} $.
	Denote by $ (y_{i}) $ 
	the cyclically ordered
	neighbours of $ x $
	in $ \ti{\TT} $,
	and $ \alpha_{i} =
	\angle_{F(x)} (F(y_{i}), F(y_{i+1})) $.
	Then:
	\begin{equation}
		\sum_{i}
		\alpha_{i}
		\ge
		2 \pi
		.
	\end{equation}
\end{propo}

The proof of proposition \ref{prop:con_angles}
use the following result of 
Izeki and Natayani:
\cite{izekiCombinatorialHarmonicMaps2005}
\begin{propo}[{\cite[prop. 3.5]{izekiCombinatorialHarmonicMaps2005}}]
	\label{prop:critic}
	If $ F : \ti{V} \to X $
	is an harmonic,
	$ \rho $-equivariant map,
	then
	for every vertex
	$ x \in \ti{V} $
	and every direction
	$ v \in \Sigma_{F(x)} X $ 
	at $ x $:
	\begin{equation}
		\sum_{
			e \in \ti{E}
		\tq e_{-} = x}
		\abs{x e_{+}}
		\cos \angle_{x} ( v, e_{+}  )
		\le 0
		.
	\end{equation}
\end{propo}
We call this inequality
the \emph{harmonic critical inequality},
because it expresses the fact
that a harmonic map
is a critical point of the energy.
It means that every vertex
is the barycenter of its
neighbours, in a weak sense.

We also need a lemma about
short geodesic polygons in $ \cat{1} $-spaces.
Recall that the radius of a set
$ A $ in a metric space is the
infimum of the radius of balls
that contain $ A $.
\begin{lemma}
	\label{lem:short_poly}
	Let $ \sigma $
	be a geodesic polygon
	of length $ < 2 \pi $
	in a $ \cat{1} $-space $ \Sigma $.
	Then the radius of
	$ \sigma $ is $ < \frac{\pi}{2} $.
\end{lemma}
\begin{proof}
	The lemma is
	true if $ \Sigma = \Sp^{2} $,
	see for example 
	\cite[Lem. 1.2.1]{alexanderAlexandrovGeometryPreliminary2019}.
	For general $ \Sigma $,
	as the length of $ \sigma $ 
	is $ < 2 \pi $,
	we apply Reshetnyak's majorization
	theorem \ref{th:majorization} to $ \sigma $.
	It gives a convex set $ C \subset \Sp^{2} $
	whose boundary $ \ti{\sigma} $
	have the same length
	as $ \sigma $,
	and a 1-Lipschitz map
	$ M : C \to \Sigma $
	which sends
	$ \ti{\sigma} $ to $ \sigma $.
	We can apply the lemma
	to $ \ti{\sigma} \subset \Sp^{2} $,
	and we have that the radius
	of $ \ti{\sigma} $ is
	$ < \frac{\pi}{2} $.
	Because $ M $ is 1-Lipschitz
	and maps 
	$ \ti{\sigma} $ to $ \sigma $,
	we have that the radius
	of $ \sigma $ is 
	$ < \frac{\pi}{2} $.
\end{proof}

We can now prove proposition
\ref{prop:con_angles}.
\begin{proof}
	[Proof of proposition \ref{prop:con_angles}]
	We work in the
	space $ \Sigma = \Sigma_{F(x)} X $.
	Recall that it is a $ \cat{1} $-space
	and that the distance in $ \Sigma $
	is given by
	$ \abs{uv} = \angle (u, v)$
	for two geodesics $ u,v $ 
	starting at $ F(x) $.
	Consider the geodesic polygon $P$
	whose
	vertices are the $ [F(x) F(y_{i}] $.
	The length of $ P $
	is
	$ \sum_{i} \alpha_{i} $.
	Suppose, for the sake of contradiction,
	that
	$ \sum_{i} \alpha_{i} < 2 \pi $.
	Then applying
	lemma \ref{lem:short_poly} to $ P $
	gives that the radius
	of $ P $ is $ < \frac{\pi}{2} $.
	It means that there exists
	a direction $ v \in \Sigma $ such that for all
	$ u \in P $,
	$ \abs{uv} < \frac{\pi}{2} $.
	Now, by proposition \ref{prop:critic}
	we have
	\begin{equation}
		\sum_{i}
		\abs{F(x) F(y_{i}}
		\cos \angle_{F(x)}  ( v, F(y_{i}) )
		\le 0
		,
	\end{equation}
	but
	$\cos \angle_{F(x)}  ( v, F(y_{i}) )$
	is positive because
	$ \angle_{F(x)}  ( v, F(y_{i}) ) < \frac{\pi}{2} $.
	This is a contradiction.
	We infer that
	$ \sum_{i} \alpha_{i} \ge 2 \pi $.
\end{proof}

\section{Triangulated conical hyperbolic surfaces}
\label{sec:con_surf}

In this section
we explain the concept
of triangulated conical hyperbolic surface
and how a metric triangulation
allows us to construct
a domination of $ \rho $
by a triangulated conical hyperbolic surface.

We use a definition suited for our
purposes:
\begin{definition}
	\label{def:con_hyp_surf}
	A \emph{triangulated conical hyperbolic surface}
	is a connected,
	triangulated metric space
	such that each face of the
	triangulation
	is isometric to
	a triangle in $ \HH^{2} $.
\end{definition}
A triangulated conical hyperbolic surface
is a
$ \HH^{2} $-simplicial complex,
in the sense of
\cite[Chapter I.7]{haefligerMetricSpacesNonpositive1999}.
We will only work with compact,
or covering of compact
triangulated
conical hyperbolic surfaces:
this implies that
such a surface has finitely
many isometry type of shapes,
and by \cite[Th. I.7.19]{haefligerMetricSpacesNonpositive1999},
it is a complete geodesic space.

For a triangulated conical hyperbolic surface $ C $,
we say that $ C $ is \emph{non-degenerate}
if each face of the triangulation
is a non-flat triangle.
In this case, $ C $ is a topological
surface.
If no edge of the triangulation
is flattened, we say that
the conical angle $ \theta_{x} $ of $ S $
at a vertex $ x $ of the triangulation
is the sum of the angles at $ x $
of the triangles adjacent to $ x $.
We have the following condition
for a triangulated conical hyperbolic surface
to be of curvature $ \le -1 $:
\begin{theorem}
	\label{th:cond_hyperbolic}
	[{\cite[th. I.7.5.2]{haefligerMetricSpacesNonpositive1999}}]
	A triangulated conical hyperbolic surface
	with no edge flattened
	is of curvature $ \le -1 $
	if and only if
	for each vertex
	$ x $ of the triangulation
	the conical angle
	$ \theta_{x} $
	is $ \ge 2 \pi $.
\end{theorem}

We explain how to
construct triangulated conical hyperbolic surfaces.
Given a topological surface $ S $
and a metric triangulation
$ (\TT, \ell) $ on $ S $
we construct a conical hyperbolic
surface:
for each face $ f $ of
$ \TT $ with sides
$ e_{1} $,
$ e_{2} $
and
$ e_{3} $,
consider the triangle
$ \Delta_{f} $ 
in $ \HH^{2} $ with side lengths given by
$ \ell (e_{1})
, \ell (e_{2})$
and $ \ell (e_{3})$.
Then glue the triangles
$ \Delta_{f} $ along
their common side
following
the combinatoric of $ \TT $.
This gives a
triangulated conical hyperbolic surface $ C_{\TT} $
and a continuous surjection
$ p_{\TT} : S \to C_{\TT} $.
If $ \ell $ does not flatten any face,
this map is a homeomorphism.
We call this triangulated conical hyperbolic surface
$ C_{\TT} $ 
a conical structure on $ S $
(even when $ p_{\TT} $ is not a
homeomorphism).

Now suppose that
we are given a topological
surface $ S $
and a metric triangulation
$ (\TT, \ell) $ on $ S $.
The fundamental group
$ \pi_{1} C_{\TT} $
acts by deck transformation
on
its fundamental group
$ \ti{C_{\TT}} $
by isometries.
This gives a representation
$ \pi_{1} C_{\TT}
\to
\Isom (\ti{C_{\TT}}) $.
Composing with the
morphism induced
on $ \pi_{1} $ 
by $ p_{\TT}  $,
this gives a representation
$ \rho_{\TT} :
\pi_{1} S
\to
\Isom (\ti{C_{\TT}}) $
.
We call it the representation
associated to the conical
structure $ C_{\TT} $
on $ S $.
When the metric triangulation
is induced by
a $ \rho $-equivariant map
$ F: \ti{V} \to X $,
we denote by $ C_{F} $
the associated conical structure
and by $ \rho_{F} $
the associated representation.

If the
associated conical structure
$ C_{\TT} $
satisfies the condition
of theorem
\ref{th:cond_hyperbolic}
(remark that this condition
depends only of $ (\TT, \ell) $),
by theorem
\ref{th:cond_hyperbolic},
the space $ C_{\TT} $ 
is of curvature $ \le -1 $.
The universal covering
$ \ti{C_{\TT}} $
is then a $ \cat{-1} $ space.

\begin{propo}
	\label{prop:harm_dom}
	Let $ \TT = (V, E, F) $ be
	a triangulation
	on $ S $ and
	$ F : \ti{V} \to X $
	a
	$ \rho $-equivariant map.
	Then
	the representation
	$ \rho_{F} $ 
	associated
	to the conical structure
	$ C_{F} $ 
	on $ S $
	given by the metric
	triangulation
	$ (\TT, \ell_{F}) $ 
	dominates
	the representation $ \rho $.
\end{propo}
\begin{proof}
	Consider the
	universal
	covering
	$ Y = \ti{C_{F}} $.
	Remark that
	the triangulation
	$ \ti{\TT} $
	induces
	a triangulation
	$ \TT^{'} = (V', E', F') $
	on $ Y $
	and that
	$ F $ factors through
	$ p_{\TT} $
	and defines a
	map
	$ V' \to X $,
	equivariant
	with respect to the deck
	transformation and $ \rho $.
	Each triangle
	$ \Delta $ of the triangulation
	is,
	by construction,
	the comparison triangle
	of the triangle $ F(\Delta) \subset X $,
	that is the triangle in $ \HH^{2} $
	with the same side lengths.
	Because $ X $ is a $ \cat{-1} $ space,
	the comparison map
	$ \Delta \to F(\Delta) $ 
	(which maps the corresponding
	sides isometrically)
	is 1-Lipschitz.
	If we extend $ F $
	on each triangle
	by the comparison map,
	equivariantly,
	then
	$ F $ defines
	a 1-Lipschitz map
	$ Y \to X $,
	equivariant
	with respesct
	to the desk tranformations
	and $ \rho $.
	Pulling back $ F $
	to $ \ti{S} $
	by $ p_{\TT} $
	gives a
	1-Lipschitz map
	$ \ti{S} \to X $ 
	which is
	$ (\rho_{F}, \rho) $-equivariant.
	It follows that
	$ \rho_{F}  $
	dominates
	$ \rho $.
\end{proof}
When we are in the hypothesis
of proposition \ref{prop:harm_dom}
we say that the
conical structure $ C_{F} $ 
dominates $ \rho $.

We study the
conformal structure
on a non-degenerate
triangulated conical hyperbolic surface.
If $ C = C_{\TT} $ is a
non-degenerate
conical structure on $ S $
with triangulation
$ \TT = (V, E, F) $,
then $ C \setminus V $
is a surface of constant
curvature $ -1 $.
It implies that it
exists a unique
conformal structure
on $ S \setminus V $
compatible
with this metric.
We equip $ C \setminus V $ 
with this structure
of Riemann surface.
It is well known that
we can extend
this structure
on all of $ C $:
\begin{propo}
	\label{prop:conf_struct}
	Using the same notations,
	there exists
	a structure of Riemann
	surface on $ C $
	compatible with
	the one on $ C \setminus V $.
	Denote by $ g $
	the metric tensor
	defining the metric
	on $ C \setminus V $.
	On each point $ z_{0} \in V $,
	there exists a holomorphic
	chart centered at $ z_{0} $
	such that in the coordinate
	$ z $ 
	given by this chart,
	$ g $ has the expression:
	\begin{equation}
		\abs{z}
		^{2(\frac{\theta}{2 \pi}-1)}
		\phi(z)
		\abs{dz}^{2} 
		,
	\end{equation}
	where $ \theta $
	is the conical angle at
	$ z_{0} $ 
	and $ \phi $
	is a positive function.
\end{propo}
Remark that when the conical
structure has curvature $ \le -1 $,
the metric tensor does not
blow up at the vertices
of the triangulation
(because $ \theta \ge 2 \pi $).

We will use the following
strong version of
Ahlfors' lemma
\cite{ahlforsExtensionSchwarzLemma1938}
due to Minda
to compare two conformal metrics:
\begin{theorem}
	[\cite{mindaStrongFormAhlfors1987}]
	\label{th:ahlfor_lemma}
	Let $ \Sigma $
	be a Riemann surface
	and $ g_{0} $ the conformal metric
	with constant curvature $ -1 $.
	If $ g $ is another conformal
	metric, allowed to vanish,
	with curvature $ \le -1 $
	at the points where it does not
	vanish then
	$
		g < g_{0}
	$
	everywhere or else
	$
		g = g_{0} 
	$
	everywhere.
\end{theorem}

\section{Proof of the theorem}
\label{sec:proof}

In this section we prove
the main theorem \ref{th:main}.
Recall that $ S $ is a topological surface,
$ X $ is a $ \cat{-1} $ space,
$ \rho : \Gamma \to \Isom (X) $
a representation from the fundamental group
of $ S $ to the isometry group of $ X $.

Start with a triangulation
$ \TT = (V, E, F) $
of $ S $
with only one vertex.
If the representation $ \rho $
fixes a point on the boundary
at infinity
$ \partial_{\infty} X $,
the argument of
\cite[sec. 2.1]{deroinDominatingSurfaceGroup2016}
allows us to conclude that
$ \rho $ satisfies the conclusion
of the theorem.
If $ \rho $ does not
fix a point of
$ \partial_{\infty} X $,
we apply
theorem \ref{th:harm_exist}
which gives the existence
of an harmonic, $ \rho $-equivariant
$ F : \ti{V} \to X $.

Consider the metric triangulation
$ (\TT, \ell_{F}) $
and conical structure
$ C_{F} $ induced on $ S $
by $ F $.
We first suppose that $ C_{F} $
is non-degenerate.
Applying proposition \ref{prop:con_angles},
we know that all the conical angles
of $ C_{F} $ are $ \ge 2 \pi $.
By theorem \ref{th:cond_hyperbolic}
this means that $ C_{F} $
is of curvature $ \le -1 $.
According to proposition
\ref{prop:harm_dom},
the representation
$ \rho_{F} $ dominates
$ \rho $.
By proposition
\ref{prop:conf_struct},
$ C_{F} $ can be equiped
with a structure
of Riemann surface
$ \Sigma $ 
such that the tensor
$ g $ defining
the conical metric
is conformal
( $ g $ is allowed to vanish
at the vertices of the triangulation)
and of curvature $ -1 $
where it does not vanish.
Let $ g_{0} $ be the conformal metric
on $ S $ with constant curvature $ -1 $.
The metrics $ g $ and $ g_{0} $
satisfies the hypothesis
of theorem \ref{th:ahlfor_lemma},
so we either have $ g < g_{0} $
or $ g = g_{0} $ everywhere.

If $ g = g_{0} $ it means that
$ C_{F} $ is a (non-conical)
hyperbolic surface and that
$ \rho_{F} $ is a Fuchsian representation.
As we have seen that $ \rho_{F} $
dominates $ \rho $, we have proved
that $ \rho $ is dominated
by a Fuchsian representation.
By proposition \ref{prop:conf_struct},
because $ g $ does not vanish,
the conical angles of $ C_{F} $
are all equals to $ 2 \pi $
and
we prove in section \ref{sec:rigidity}
that it implies that the
domination extends to an isometric embedding.

If $ g < g_{0} $,
it means that the identity map
$ (S, g_{0}) \to (S, g) $
is c-Lipschitz
for some $ c < 1 $.
Denoting by $ j $ the Fuchsian
representation which is the
holonomy of the hyperbolic surface
$ \Sigma $, it means that
$ j $ strictly dominates $ \rho_{F} $.
Composing the domination
between $ j $ and $ \rho_{F} $
and between $ \rho_{F} $ and
$ \rho $ gives a domination
between $ j $ and $ \rho $.
This domination is strict
because the one between
$ j $ and $ \rho_{F} $
is strict.

We have proved the theorem,
assuming that $ C_{F} $
is non-degenerate.
We treat the case where
$ C_{F} $ is degenerate
in section
\ref{sec:desing}.

\section{Desingularization}
\label{sec:desing}

In this section we show
that given a degenerate
conical structure
$ C_{\TT} $
on $ S $,
with either no edges flattened
and conical angle $ \ge 2 \pi $
or some edge flattened,
we have two cases:
etiher
we can
find a non-degenerate
conical structure
$ C' $ on $ S $
which dominates
$ C_{\TT} $,
whose conical angle
is $ > 2 \pi $ 
(and the domination preserves the triangulation)
or no edge is flattened,
the conical angle is $ 2 \pi $
and exactly one face is flattened.
In the first case,
we can apply the argument
of section
\ref{sec:proof}
to $ C' $
and in the second case,
we can apply directly apply
the rigidity argument
of section
\ref{sec:rigidity}
to $ C_{\TT} $.
In both cases, this proves
theorem
\ref{th:main}.

Denote by $ (\TT, \ell) $
the metric triangulation
defining the conical structure.
The idea is to perturb
$ \ell $ so that
it becomes non-degenerate.
First we state a result
which allows us to
perturb individual triangles.
\begin{propo}
	\label{prop:desing_triangles}
	Let $ \Delta $
	be an hyperbolic triangle
	with sides
	$ (l_{1},
	l_{2}, 
	l_{3})  $
	and $\varepsilon > 0$.
	Define $ \Delta_{\varepsilon} $
	to be the hyperbolic triangle
	with sides
	$ (l_{1}+\varepsilon,
	l_{2}+\varepsilon, 
	l_{3}+\varepsilon)  $
	Then there exists
	a 1-Lipschitz
	map $ \Delta_{\varepsilon} \to \Delta $
	which maps vertex to corresponding vertex
	and side to corresponding side.
	The restriction of this map
	to each side depends only
	of the length of this side
	and $ \varepsilon $.
\end{propo}
We say that $ \Delta_{\varepsilon} $
dominates $ \Delta $.
\begin{proof}
	Let $ x,y,z $
	be the vertices of $ \Delta $.
	Consider the space
	$ \Delta_{\varepsilon}^{*} $ 
	obtained by gluing
	a segment of length $ \varepsilon/2 $
	to each vertex of $ \Delta $
	(see figure \ref{fig:triangle_epsilon}).
	Call $ x',y',z' $ the
	extremities of these segments.
	By Reshetnyak's gluing theorem
	this is a $ \cat{-1} $ space.
	The map that crushes
	these segments to the corresponding
	vertex is a 1-Lipschitz
	$ \Delta_{\varepsilon}^{*} \to \Delta $,
	because it is the projection
	on the convex $ \Delta \subset \Delta_{\varepsilon}^{*} $.

	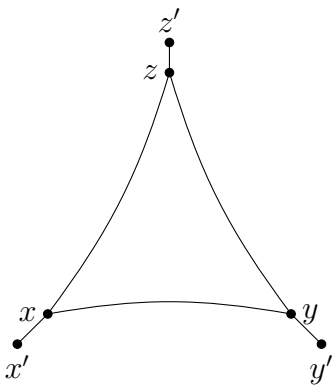
\begin{figure}[h]
		\centering
		\begin{tikzpicture}[scale=0.8]
\draw (0,0) node[anchor=east]{$x$}
to[bend left=10] (4,0) node[anchor=west]{$y$}
to[bend left=10] (2, 4) node[anchor=east]{$z$}
to[bend left=10] cycle;
\draw (0,0)
	-- (-0.5,-0.5) node[anchor=north]{$x'$};
\draw (4,0)
	-- (4.5,-0.5) node[anchor=north]{$y'$};
\draw (2,4)
	-- (2,4.5) node[anchor=south]{$z'$};
\filldraw [black] (0,0) circle (2pt);
\filldraw [black] (4,0) circle (2pt);
\filldraw [black] (2,4) circle (2pt);
\filldraw [black] (-0.5,-0.5) circle (2pt);
\filldraw [black] (4.5,-0.5) circle (2pt);
\filldraw [black] (2,4.5) circle (2pt);

\end{tikzpicture} 
		\caption{The triangle $ \Delta_{\varepsilon}^{*} $ }
		\label{fig:triangle_epsilon}
	\end{figure}

	Now, consider the triangle
	with vertices $ x',y',z' $
	in $ \Delta_{\varepsilon}^{*} $.
	This is a triangle with sides
	$ (l_{1}+\varepsilon,
	l_{2}+\varepsilon, 
	l_{3}+\varepsilon) $.
	Applying the $ \cat{-1} $
	inequality
	we have that
	the comparison map
	between the comparison
	hyperbolic triangle
	$ \Delta_{\varepsilon} $ 
	and the
	triangle
	$ \Delta_{\varepsilon}^{*} $
	is
	1-Lipschitz.
	Composing the comparison map
	$ \Delta_{\varepsilon} \to
	\Delta_{\varepsilon}^{*} $
	and the crushing map
	$ \Delta_{\varepsilon}^{*}
	\to
	\Delta$
	we get the desired 1-Lipschitz
	map
	$ \Delta_{\varepsilon} \to \Delta $.
	Its restriction to each side
	is the projection
	on the convex subset
	$ [\varepsilon/2, l_{i}  + \varepsilon/2] $ 
	of
	$ [0, l_{i}  + \varepsilon] $.

\end{proof}
Observe that if we apply
this proposition to a flat
triangle $ \Delta $,
then $ \Delta_{\varepsilon} $ is not flat
for any $ \varepsilon > 0 $.
Also, the angles of $ \Delta_{\varepsilon} $
are as close as desired to the angles
of $ \Delta $ when $ \varepsilon $
is small enough.
If we apply this proposition
to every triangle of
the triangulation we get:
\begin{propo}
	\label{prop:desing_glob}
	Let $ (\TT, \ell) $
	be a metric triangulation,
	with associated conical structure
	$ C $,
	$\varepsilon > 0$
	and
	let $ (\TT, \ell_{\varepsilon}) $
	be the metric triangulation
	where $ \ell_{\varepsilon} = \ell + \varepsilon $,
	with associated conical structure
	$ C' $.
	Then
	$ (\TT, \ell_{\varepsilon}) $
	does not flatten any triangle
	and we have a
	1-Lipschitz map
	$ C' \to C $
	.
\end{propo}

\subsubsection{No edge flattened}
\label{ssub:no_edge_flattened}

We first suppose that
$ C = C_{F} = C_{\TT} $ is degenerate
but does not
flatten edges.
Remember that the triangulation
has only one vertex.
Because $ C_{\TT} $
is degenerate,
some triangle
is flattened.
We separate two cases:
the conical angle
is
$ > 2 \pi $
or it is $ 2 \pi $.

Consider the first case.
We apply proposition
\ref{prop:desing_glob}
to $ (\TT, \ell) $
with some $ \varepsilon > 0 $.
The conical structure
$ C' $ associated
to the metric triangulation
$ (\TT, \ell_{\varepsilon}) $
is non-degenerate
and we have
a $ 1 $-Lipschitz
map $ C' \to C $.
This means that
$ C' $ dominates $ C $.
If $ \varepsilon $ is small
enough, the conical angle
of $ C' $ is close enough
to the conical angle of $ C $,
so it is $ > 2 \pi $.
This proves the result in this case.

Consider the second case.
A flat triangle
has angles
$ (\pi, 0, 0) $.
As all of these
angles meet at
the only vertex,
they contribute
to the conical angle
with $ \pi $.
If we had at least
2 flat triangles,
they would
contribute to the conical angle
with $ 2 \pi $.
As there is at least
one other triangle
which contribute
with a positive angle,
the conical angle
would be $ > 2 \pi $.
As we assumed
that the conical angle
is $ 2 \pi $,
it means that there is
only one flat triangle.
We can apply the
rigidity argument of
section
\ref{sec:rigidity}
to $ C_{\TT} $.

\subsubsection{Some edge flattened}
\label{ssub:some_edge_flattened}

Suppose now that
$ C $ is degenerate,
and flatten some edges.
We can suppose that
at least one edge
is not flattened.
Otherwise, it means that
$ F $ is constant,
and the unique point in its image
is fixed by the representation $ \rho $.
In this case, $ \rho $ is trivially dominated
by any Fuchsian representation.

Because at least one edge is not flattened,
there exists a
triangle $ \Delta $
with side lengths
$ (a, a, 0) $ for some
$ a > 0 $.
We apply proposition
\ref{prop:desing_glob}
to $ (\TT, \ell) $
with some $ \varepsilon > 0 $.
The conical structure
$ C' $ associated
to the metric triangulation
$ (\TT, \ell_{\varepsilon}) $
is non-degenerate
and we have
a $ 1 $-Lipschitz
map $ C' \to C $.
This means that
$ C' $ dominates $ C $.
It remains to
show that the conical angle
of $ C' $ is $ > 2 \pi $.

The triangle $ \Delta $
with sides $ (a,a,0) $
in $ (\TT, \ell) $
gives the triangle
$ \Delta_{\varepsilon} $ 
with sides
$ (a+\varepsilon,a+\varepsilon,\varepsilon) $
in $ (\TT, \ell_{\varepsilon}) $.
An application of the
hyperbolic law of cosines
show that when $ \varepsilon $
goes to 0, the angles
of $ \Delta_{\varepsilon} $ 
go to $ (\frac{\pi}{2}, \frac{\pi}{2}, 0) $.
So the triangle $ \Delta_{\varepsilon} $
contibutes with $ \pi + o(1) $
to the conical angle of $ C' $,

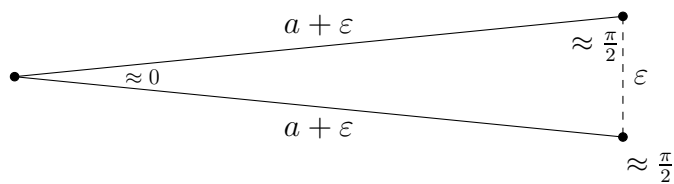
\begin{figure}[h]
	\centering
	\begin{tikzpicture}[scale=0.8]
	\coordinate (P) at (0,0);
	\coordinate (Q) at (10,1);
	\coordinate (R) at (10,-1);

	\draw (P)
	-- node[above]{$a+\varepsilon$}
	(Q);
	\draw[dashed] (Q)
	-- node[right]{$\varepsilon$}
	 (R);
	\draw (R)
	-- node[below]{$a+\varepsilon$}
	(P);

	\tkzLabelAngle[pos = 0.75, scale=0.9](P,Q,R){$\approx \frac{\pi}{2}$}

	\tkzLabelAngle[pos = 0.75, scale=0.9](P,R,Q){$\approx \frac{\pi}{2}$}

	\tkzLabelAngle[pos = 3, scale=0.7](R,P,Q){$\approx 0$}

	\filldraw [black] (P) circle (2pt);
	\filldraw [black] (Q) circle (2pt);
	\filldraw [black] (R) circle (2pt);

\end{tikzpicture} 
	\caption{The perturbation of a triangle with an edge flattened}
	\label{fig:edge_flattened}
\end{figure}

Consider the triangle $ \Delta' $
which share with $ \Delta $ the
edge of length $ 0 $.
We have two case: either
$ \Delta $ has sides
$ (0,0,0) $
or it has sides $ (b,b,0) $
with $ b > 0 $.
If it has sides 
$ (0,0,0) $,
then $ \Delta'_{\varepsilon} $
has sides
$(
\varepsilon,
\varepsilon,
\varepsilon
)$,
so its angles
go to
$(
\frac{\pi}{3},
\frac{\pi}{3},
\frac{\pi}{3}
)$
when $ \varepsilon $
goes to 0.
The triangle $ \Delta'_{\varepsilon} $ 
contributes
$ \pi - o(1) $
to the conical angles of $ C' $.
If it has sides $ (b,b,0) $,
we argue as in the previous
paragraph,

Consider the triangle
$ \Delta^{''} $
which share with $ \Delta $
one of the two
edge of length $ a $.
The sum of the angles
of this triangle is
some $ \alpha > 0 $.
The triangle $ \Delta^{''}_\varepsilon $
has the sum of its angles
which goes to $ \alpha $
when $ \varepsilon $
goes to 0.
So 
the triangle $ \Delta^{''}_\varepsilon $
contributes with $ \alpha + o(1) $
to the conical angle of $ C' $.

Summing all of these contributions
we have that the conical angle
of $ C' $ is at least
$ 2 \pi + \alpha + o(1) $
when $ \varepsilon $ goes to
0.
For $\varepsilon$
small enough,
this quantity is $ > 2 \pi $.

\section{Rigidity case}
\label{sec:rigidity}

We first assume
that $ X $ is proper,
that is the closed ball
are compact.
We suppose that we are in the case
where $ C_{F} $ is either non-degenerate,
or flatten at most one face and no edge,
and that the conical angle is $ 2 \pi $.
Let $ \xi $ be one vertex in the
triangulation lifted to the universal cover
and $ (\eta_{i}) $ its neighbours,
ordered such that the faces of the triangulations
adjacent to $ \xi $ 
are $ [\eta_{i} \xi \eta_{i+1}]$
(we consider the indices $ i $
cyclically).
We denote $ x = F(\xi) $ and $ y_{i} = F(\eta_{i}) $.

The conical angle of
the conical hyperbolic surface
equals $ 2 \pi $,
that is the sum of the comparison
angles of the triangles at $ x $
equals $ 2 \pi $.
We show that in this case, $ F $
is in fact an isometry.

Let 
$\Delta_{i} = [x, y_{i}, y_{i+1}]$
be the triangles adjancent to $ x $.
We denote
$\alpha_{i} =
\angle_{x}
(y_{i}, y_{i+1})$
and $\ti{\alpha_{i}}
=
\ti{\angle}_{x}^{(-1)} 
(y_{i}, y_{i+1})
$
the corresponding comparison angle.
Our hypothesis is
\begin{equation}
	\sum_{i}
	\ti{\alpha_{i}}
	=
	2 \pi
	.
\end{equation}
Because $ F $ is harmonic we have
$
\sum_{i}
\alpha_{i}
\ge
2 \pi
$
and we always have
$\ti{\alpha_{i}} \ge \alpha_{i}$.
Then:
\begin{equation}
	2 \pi
	=
	\sum_{i}
	\ti{\alpha_{i}}
	\ge
	\sum_{i}
	\alpha_{i}
	\ge
	2 \pi
	,
\end{equation}
and all inequalities become
equalities, so
$ \alpha_{i} = \ti{\alpha_{i}} $.
We are in the rigidity case
of the $ \cat{-1} $ inequality
for $ \Delta_{i} $ so the
comparison map from
$ \ti{\Delta_{i}} $
to $ \Delta_{i} $
is an isometry.
By construction, $ F $
is an isometry in restriction
to each $ \Delta_{i} $.

Now, it is enough
to show that
the union of triangles
$ \Delta_{i} \cup \Delta_{i+1} $
in $ X $ 
is isometric via $ F $ to the union
of the comparison triangles
$ \ti{\Delta_{i}} \cup \ti{\Delta_{i+1}} $
in $ \HH^{2} $,
as it would imply that $ F $
is a global isometry.

We show this result,
assuming first that
\begin{equation}
	\alpha_{i} + \alpha_{i+1}
	=
	\angle_{x}
	(y_{i}, y_{i+2})
	.
\end{equation}

As $ F $ is an isometry
from
$ \ti{\Delta_{j}} $,
to $ \Delta_{j} $
it is enough to show that
for each
$ z=F(\ti{z})
\in  \Delta_{i} $ et
$ z'=F(\ti{z'})
\in  \Delta_{i+1} $
the distance $ \abs{z z'} $
in $ X $ 
equals the distance
$ \abs{\ti{z} \ti{z'}}$
in
$ \ti{\Delta_{i}}
\cup
\ti{\Delta_{i+1}} $.

First we compute
$\angle_{x}
(z, z')$.
We have:
\begin{align}
	\angle_{x}
	(y_{i}, y_{i+2})
	& \le
	\angle_{x}
	(y_{i}, z)
	+
	\angle_{x}
	(z, z')
	+
	\angle_{x}
	(z', y_{i+2}) \\
	& \le
	\angle_{x}
	(y_{i}, z)
	+
	\angle_{x}
	(z,y_{i+1})
	+
	\angle_{x}
	(y_{i+1},z')
	+
	\angle_{x}
	(z', y_{i+2}) \\
	& =
	\angle_{x}
	(y_{i}, y_{i+1})
	+
	\angle_{x}
	(y_{i+1}, y_{i+2}) \\
	& =
	\angle_{x}
	(y_{i}, y_{i+2})
	,
\end{align}
where we used the fact that 
the triangles $ \Delta_{j} $
are isometric to
hyperbolic triangles
and so that
\begin{equation}\angle_{x}
	(y_{i}, z)
	+
	\angle_{x}
	(z,y_{i+1})
	=
	\angle_{x}
	(y_{i}, y_{i+1})
\end{equation}
and
\begin{equation}\angle_{x}
	(y_{i+1}, z')
	+
	\angle_{x}
	(z',y_{i+2})
	=
	\angle_{x}
	(y_{i+1}, y_{i+2})
	.
\end{equation}
Every intermediate inequality
is then an equality and so
$
\angle_{x}
(z, z')
=
\angle_{x}
(z,y_{i+1})
+
\angle_{x}
(y_{i+1},z')
$.

\begin{figure}[h]
	\centering
	\begin{tikzpicture}[scale=0.8]
	\coordinate (X) at (0,0);
	\coordinate (Y1) at (-2,2);
	\coordinate (Y2) at (0,6);
	\coordinate (Y3) at (2,3);

	\coordinate (Z) at (-1,3);
	\coordinate (Z') at (1,2.5);

	\draw (X)
	--
	(Y1);
	\draw (X)
	--
	(Y2);
	\draw (X)
	--
	(Y3);
	\draw (Y1)
	--
	(Y2);
	\draw (Y2)
	--
	(Y3);

	\draw[dashed] (X)
	--
	(Z);
	\draw[dashed] (X)
	--
	(Z');
	\draw[dashed] (Z)
	--
	(Z');

	\filldraw [black] (X)
	node[anchor=north]{$x$}
	circle (2pt);
	\filldraw [black] (Y1)
	node[anchor=east]{$y_{1}$}
	circle (2pt);
	\filldraw [black] (Y2)
	node[anchor=south]{$y_{2}$}
	circle (2pt);
	\filldraw [black] (Y3)
	node[anchor=west]{$y_{3}$}
	circle (2pt);
	\filldraw [black] (Z)
	node[anchor=east]{$z$}
	circle (2pt);
	\filldraw [black] (Z')
	node[anchor=west]{$z'$}
	circle (2pt);

\end{tikzpicture} 
	\caption{The triangles $ \Delta_{1} $ and $ \Delta_{2} $ }
	\label{fig:isom_2_trig}
\end{figure}
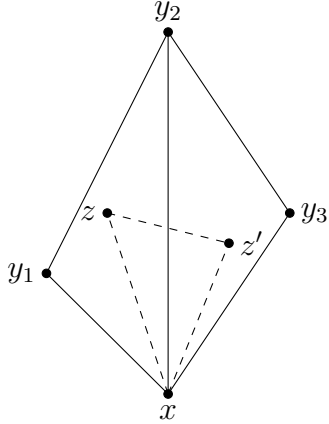

Now consider the triangle
$ [\ti{z}, \ti{x}, \ti{z'}] $
in $ \HH^{2} $.
Its angle at $ \ti{x} $ is
$
\angle_{x}
(z,y_{i+1})
+
\angle_{x}
(y_{i+1},z')
=
\angle_{x}
(z, z')
$.
The comparison triangle
$ \ti{\Delta} $
of $ (z, x, z') $ is an 
hyperbolic triangle
whose 
adjacents sides at $ x $ 
are of the same sizes
as the one of
$ [\ti{z}, \ti{x}, \ti{z'}] $
at $ \ti{x} $
and whose angle at $ \ti{x} $
is greater or equal than
$
\angle_{x}
(z, z')
$.
As the opposite side of an
hyperbolic triangle is an 
increasing function of the
angle
we have
$ \abs{\ti{z} \ti{z'}} \le \abs{z z'} $
and so
$ \abs{\ti{z} \ti{z'}} = \abs{z z'} $.
We conclude that $ F $ is an
isometry from
$ \ti{\Delta_{i}} \cup \ti{\Delta_{i+1}} $
to
$ \Delta_{i} \cup \Delta_{i+1} $.

We now show what was assumed
until then, that
\begin{equation}
	\label{eq:eq_cons_angles}
	\alpha_{i} + \alpha_{i+1}
	=
	\angle_{x}
	(y_{i}, y_{i+2})
	.
\end{equation}
The strategy is to interpret this equality
in the space of directions at $ x $.
Let $ p_{i} = [x y_{i}] \in \Sigma_{x} X $
the image of $ y_{i} $ in the space of directions
at $ x $.
By definition of the distance in the space
of directions, the inequalities we want to prove
are
\begin{equation}
	\abs{p_{i} p_{i+1}}
	+
	\abs{p_{i+1} p_{i+2}}
	=
	\abs{p_{i} p_{i+2}}
	,
\end{equation}
that is the points $ p_{i} $, $ p_{i+1} $
and $ p_{i+2} $ are aligned on a geodesic.
Another way to say that is to considere the
geodesic polygon
$ P = [p_1, \dots, p_{N} ]$.
The equalities we want to prove are equivalent
to the fact that $ P $ is a local geodesic.
That's what we are going to prove.

We argue by contradiction
and suppose that $ P $ is not
a local geodesic, that is
we have
$
\abs{p_{1} p_{2}}
+
\abs{p_{2} p_{3}}
<
\abs{p_{1} p_{3}}
$
(up to renaming).
We say that
$ P $ has an angle at $ p_{2} $.
To obtain a contradiction, we will
use a comparison polygon $ \ti{P} $
in $ \Sp^{2} $. The harmonicity
of the triangulation will force
this polygon to be a great circle,
but the fact that $ P $ has an angle
will force $ \ti{P} $ to also have an
angle, which is absurd.

\subsubsection{Construction of the comparison polygon}

To construct the comparison polygon,
we want to use Reshetnyak's Majorization
theorem.
By hypothesis, $ P $ is a geodesic polygon
with perimeter $ 2 \pi $ so we can't
use the theorem directly.
Using the fact that it has an angle,
we will deform it slightly to obtain
a polygon of perimeter less than $ 2 \pi $.

Let $ T $ be the triangle
$ [p_{1} p_{2} p_{3}] $.
We fix a small $ \varepsilon > 0 $.
Let $ q_{1,\varepsilon} $
be the point on $ [p_{1} p_{2}] $
at distance $ \varepsilon $
from $ p_{2} $,
$ q_{3,\varepsilon} $
the point on $ [p_{3} p_{2}] $
at distance $ \varepsilon $
from $ p_{2} $.
and $ p_{2, \varepsilon} $
the middle point of
$ [q_{1, \varepsilon} q_{2, \varepsilon}] $.
We denote $ T_{\varepsilon} $
the triangle
$ [p_{1} p_{2,\varepsilon} p_{3}] $.

\begin{figure}[h]
	\centering
	\begin{tikzpicture}[scale=0.8]
	\coordinate (P1) at (0,0);
	\coordinate (P2) at (5, 5);
	\coordinate (P3) at (10, 0);

	\coordinate (Q1e) at ($(P2) !0.3! (P1)$);
	\coordinate (Q3e) at ($(P2) !0.3! (P3)$);

	\coordinate (P2e) at ($(Q1e) !0.5! (Q3e)$);

	\draw[dashed] (P1)
	-- 
	(P2);
	\draw[dashed](P2)
	-- 
	(P3);

	\draw[draw=none] (Q1e)
	-- node[above left]{$\varepsilon$}
	(P2);

	\draw[draw=none] (Q3e)
	-- node[above right]{$\varepsilon$}
	(P2);

	\draw (P1)
	--
	(P2e);
	\draw (P3)
	--
	(P2e);

	\draw[dashed] (Q1e)
	--
	(Q3e);

	\filldraw [black] (P1)
	node[anchor=east]{$p_{1}$}
	circle (2pt);
	\filldraw [black] (P2)
	node[anchor=south]{$p_{2}$}
	circle (2pt);
	\filldraw [black] (P3)
	node[anchor=west]{$p_{3}$}
	circle (2pt);

	\filldraw [black] (Q1e)
	node[anchor=east]{$q_{1,\varepsilon}$}
	circle (2pt);
	\filldraw [black] (Q3e)
	node[anchor=west]{$q_{3, \varepsilon}$}
	circle (2pt);

	\filldraw [black] (P2e)
	node[anchor=south]{$p_{2, \varepsilon}$}
	circle (2pt);

\end{tikzpicture} 
	\caption{The triangle $ T_{\varepsilon} $ }
	\label{fig:p_eps}
\end{figure}
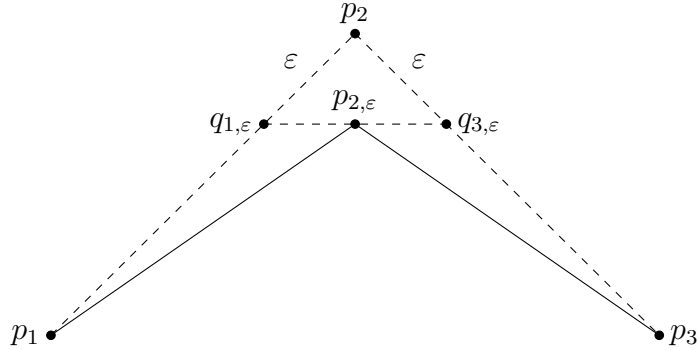

By the first variation's formula we
have
\begin{equation}
	\abs{
		q_{1, \varepsilon}
		q_{3, \varepsilon}
	}
	\sim
	\varepsilon
	\sin (\alpha)
	,
\end{equation}
when $ \varepsilon $ goes
to $ 0 $,
where $ \alpha $ is the angle
$ \angle_{p_{2} } (p_{1},p_{3}) $
(in $ \Sigma_{x} X $ ),
which is less than $ \pi $
because we assumed that $ P $
has an angle at $ p_{2} $.
We have, for
$ \varepsilon $ small enough,
\begin{align}
	\abs{p_{1} p_{2,\varepsilon}}
	& \le
	\abs{p_{1} q_{1, \varepsilon}}
	+
	\abs{q_{1, \varepsilon} p_{2,\varepsilon} }
	\\
	& \le
	\abs{p_{1} p_{2}} - \varepsilon
	+ \frac{2}{3} \sin (\alpha) \varepsilon
	\\
	& =
	\abs{p_{1} p_{2}}
	- (1 - \frac{2}{3} \sin (\alpha)) \varepsilon
	\\
	& <
	\abs{p_{1} p_{2}}
	,
\end{align} 
and,
\begin{equation}
	\abs{p_{3} p_{2, \varepsilon}}
	\le
	\abs{p_{3} p_{2}}
	-
	(1 - \frac{2}{3} \sin(\alpha)) \varepsilon
	<
	\abs{p_{3} p_{2}}
	.
\end{equation} 
Also,
$ \abs{p_{2} p_{2, \varepsilon}}
\le 2 \varepsilon $
so $ p_{2, \varepsilon} \to p_{2}  $
as $ \varepsilon $ goes to $ 0 $.
So the perimeter of $ T_{\varepsilon} $
goes to the one of $ T $
as $\varepsilon$ goes to $ 0 $,
while being smaller.

Let $ P_{\varepsilon} $ be the polygon
$ P $ with $ T $ replaced by $ T_{\varepsilon} $,
that is the polygon
$ [p_{1}, p_{2, \varepsilon}, p_{3}, \dots, p_{N}] $.
It is an approximation of $ P $ of perimeter
less than $ 2 \pi $.
We apply the Reshetnyak's Majorization theorem to
$ P_{\varepsilon} $.
We get
\begin{itemize}
	\item 	a polygon
		$ \ti{P_{\varepsilon}}
		= [\ti{p_{1, \varepsilon}},
		\ti{p_{2, \varepsilon}},
		\dots, \ti{p_{N, \varepsilon}}]$
		in
		$ \Sp^{2} $,
		whose sides' length are the same
		as $ P_{\varepsilon} $'s,
	\item 	a convex $ D_{\varepsilon} $
		in $ \Sp^{2} $ 
		which is bounded by
		$ \ti{P_{\varepsilon}} $,
	\item 	the majorizer
		$ M_{\varepsilon}
		:D_{\varepsilon} \to \Sigma_{x} X $,
		a $ 1 $-Lipschitz map which is
		length-preserving on $ \ti{P_{\varepsilon}} $ 
		and which maps $ \ti{p_{i,\varepsilon}} $
		to $ p_{i} $, except for $\ti{p_{2, \varepsilon}}$
		which is mapped to $ p_{2, \varepsilon} $.
\end{itemize}
Because $ \ti{P_{\varepsilon}} $ is of perimeter
less than $ 2 \pi $, its radius is less than
$ \frac{\pi}{2} $.
We normalize the $ D_{\varepsilon} $ using an isometry
so that that their centers coincide as $ \varepsilon $
vary.
By compactness of $ \Sp^{2} $ and properness
of $ \Sigma_{x} X $, we can extract limits
as $ \varepsilon $ goes to $ 0 $ and obtain:
\begin{itemize}
	\item 	a polygon
		$ \ti{P}
		= [\ti{p_{1}},
		\ti{p_{2}},
		\dots, \ti{p_{N}}]$
		in
		$ \Sp^{2} $,
		whose sides' length are the same
		as $ P $'s,
	\item 	a convex $ D $
		in $ \Sp^{2} $ 
		which is bounded by
		$ \ti{P} $,
	\item 	the majorizer
		$ M:D \to \Sigma_{x} X $,
		a $ 1 $-Lipschitz map which is
		length-preserving on $ \ti{P} $ 
		and which maps $ \ti{p_{i}} $
		to $ p_{i} $.
\end{itemize}

\subsubsection{The comparison polygon lies on a great circle}

The perimeter of $ \ti{P} $ is $ 2 \pi $
so it is contained in an half-hemisphere
centered at a point $ \ti{q} $, that is
for all $ \ti{p} \in \ti{P} $ we have
$ \abs{\ti{p} \ti{q}} \le \frac{\pi}{2} $.
Applying the map $ M $ we have
$ \abs{p_{i} q} \le \frac{\pi}{2} $
(where $ q = M (\ti{q}) $),
that is
$
\angle_{x} (y_{i}, \gamma)
\le
\frac{\pi}{2}
$,
where $ \gamma $ is a geodesic
representative of $ q $.
Now if we apply the harmonic critical inequality
\ref{prop:critic}
to $ \gamma $ we obtain:
\begin{equation}
	\sum_{i}
	\abs{x y_{i}}
	\cos
	\angle_{x} (y_{i}, \gamma)
	\le
	0
	,
\end{equation}
but the sum is also positive because
$
\angle_{x} (y_{i}, \gamma)
\le
\frac{\pi}{2}
$,
so each cosine is null and
$
\angle_{x} (y_{i}, \gamma)
=
\frac{\pi}{2}
$.
This implies that
$ \abs{\ti{p_{i}} \ti{q}}=\frac{\pi}{2}$,
and so $ \ti{P}$ lies on the great circle
centered at $ \ti{q} $.

\subsubsection{The comparison polygon has an angle}
We present a slightly modified proof of the
Reshetnyak's Majorization
theorem, which gives the fact that if
the polygon has an angle, the comparison polygon
can be constructed with an angle.
We follow closely
\cite[Th. 8.12.14]{alexanderAlexandrovGeometryPreliminary2019}.
Let's recall the statement of the theorem:

\begin{theorem}
	Let
	$ P = [p_{1} \dots p_{N}] $
	be a closed polygon
	($ p_{1} = p_{N} $)
	with geodesic side in
	a $ \cat{\kappa} $ space $ Y $
	and perimeter $ < 2 D_{\kappa} $.
	Then there exists
	\begin{itemize}
		\item 	a polygon
			$ \ti{P}
			= [\ti{p_{1}}
			\dots \ti{p_{N}}]$
			in
			$ \MM_{\kappa} $,
			whose sides' length are the same
			as $ P $'s,
		\item 	a convex $ D $
			in $ \MM_{\kappa} $ 
			which is bounded by
			$ \ti{P} $,
		\item 	a majorizer
			$ M:D \to Y $,
			a $ 1 $-Lipschitz map which is
			length-preserving on $ \partial \ti{P} $ 
			and which maps $ \ti{p_{i}} $
			to $ p_{i} $.
	\end{itemize}
\end{theorem}

First we outline
the un-modified proof
of the theorem.
The proof is induction of $ N $, the number
of vertices of the polygon $ P $.
For $ N=3 $, $ P $ is a triangle
and there is a construction,
the details of which are not
important to us,
which gives
a majorizer with $ D $ which is taken
to be the convex hull of the comparison triangle
in $ \MM_{\kappa} $ and which extends the
natural comparison map.

Let's suppose that the property is true
for all number less than $ N $.
We consider the polygon $ P $ as the union
of $ T =[p_{1} p_{2} p_{3}] $
and $ Q = [p_{1} p_{3} \dots p_{N} ] $ 
glued along the side $ [p_{1} p_{3} ] $.
By the induction hypothesis,
we can majorize $ T $ and $ Q $ by
polygons $ \ti{T} $ and $ \ti{Q} $
in $ \MM_{\kappa} $.
If the gluing of $ \ti{T} $ and
$ \ti{Q} $ along
$ [\ti{p_{1} } \ti{p_{3}}] $ is a convex
subset of $ \MM_{\kappa} $, we're done.
Otherwise, consider this gluing as
new as metric space, with the induced
length-metric.
By Reshetnyak's gluing theorem, this
space is $ \cat{\kappa} $.
Because the gluing of $ \ti{T} $ and
$ \ti{Q} $ along
$ [\ti{p_{1} } \ti{p_{3}}] $ is not convex,
the gluing of $ \ti{T} $ and
$ \tau $ along 
$ [\ti{p_{1} } \ti{p_{3}}] $ is not convex
with respect to the ambient metric of $ \MM_{\kappa} $,
where
$ \tau $ is either the triangle
$ [\ti{p_{1}} \ti{p_{3}} \ti{p_{4}}] $
or the triangle
$ [\ti{p_{N}} \ti{p_{1}} \ti{p_{3}}] $
So in the the length metric of the gluing
$ \ti{T} \cup \ti{Q} $, 	
the polygon $ \ti{ T } \cup \tau $
is a
triangle. Consequently,
the polygon 
$ \ti{T} \cup \ti{Q} $,
is a $ (N-1) $-gon, and by induction
hypothesis we can majorize $ P $.

Now, let's suppose that $ P $ has
an angle at $ p_{2} $, that is the
triangle $ T = [p_{1} p_{2} p_{3} ]$ 
is not flat.
We re-follow the proof to see what
is the corresponding triangle in
$ \ti{P} $.
As before, we majorize $ T $ and $ Q $
by $ \ti{T} $ and $ \ti{Q} $.
If the gluing $ \ti{T} \cup \ti{Q} $ 	
is convex, the construction is finished
and the triangle corresponding to $ T $
is $ \ti{T} $, its comparison triangle.
As $ T $ is not flat, its comparison triangle
is also not flat, and $ \ti{P} $ has an angle
at $ \ti{p_{2} } $.
If the gluing is not convex, then
$ \ti{T} \cup \tau_{1} $ is a triangle
in the length-metric of $ \ti{T} \cup \ti{Q}$
(where we denote by $ \tau_{1} $
what we denoted by $ \tau $ earlier)
and we use the induction hypothesis.
Let's call \emph{depth} of the induction,
and denote it by $ d $,
the number of times we have to use the
induction hypothesis until we obtain a convex gluing.
Then it is clear that the triangle corresponding
to $ T $ in the final polygon we obtain is
the comparison triangle of
$ \ti{T} \cup \tau_{1} \cup \dots \cup \tau_{d} $
(considering it a triangle for the length-metric
of the gluing),
where the $ \tau_{i}$ are obtained as
$ \tau $ in the previous paragraph,
and so are triangles whose
vertices are vertices of $ Q $.
As this triangle is not flat, its comparison triangle
is also not flat and $ \ti{P} $ has an angle.
We also note that without knowing the depth and
which comparison triangle correspond to $ T $,
there is a finite number of possibilites
of such triangles.

Now we show that in our case, $ \ti{P} $ has an
angle. The theorem does not apply directly
because its perimeter is $ 2 \pi $.
We use the construction $ P_{\varepsilon} $
of the previous paragraph.
It is a polygon of length less than $ 2 \pi $
which has an angle, so applying the modified
version of the theorem, we can construct
a majorizer whose boundary $ \ti{P_{\varepsilon}} $ 
has an angle.
It remains to show that the limit $ \ti{P} $
when $ \varepsilon $ goes to $ 0 $
still as an angle.
According to the discussion at the end of
the previous paragraph,
the triangle corresponding to the vertex
where $ \ti{P_{\varepsilon}} $ has an angle
is the comparison triangle
of a triangle of the form
$ \ti{T_{\varepsilon} }
\cup \tau_{1} \cup \dots \cup \tau_{d_{\varepsilon} } $
where $ d_{\varepsilon} $ is the depth for $ P_{\varepsilon} $.
The $ \tau_{i} $ that appears are sub-triangle of the
polygon $ Q = [p_{1} p_{3} \dots p_{N}] $,
which does not depend of $ \varepsilon $.
Because $ T_{\varepsilon}  $ converges to $ T $,
all of these triangles converges to a non-flat
triangle and
the angle of $ \ti{P_{\varepsilon}} $ is bounded
away from $ 0 $.
We can conclude that $ \ti{P} $ has an angle.

\textbf{Conclusion.}
Assuming that $ P $ is not a local geodesic
we have proved that $ \ti{P} $ is a great circle
and that $ \ti{P} $ has an angle.
This is absurd, and we conclude that
$ P $ is a local geodesic, and that
the equalities \ref{eq:eq_cons_angles} hold.
Finally, $ F $ induces an isometry
on the reunion of two adjancent
triangles $ \Delta_{i} \cup \Delta_{i+1} $
and then $ F $ is a (global)
isometric embedding.

\subsection{Non-proper X}
\label{sub:non_proper}

We explain how to adjust
the argument when $ X $
is not proper.
The only place where we used
the properness of $ X $
is to construct a majorizer
of $ P $ by taking a limit
of majorizer of $ P_{\varepsilon} $.
If $ X $ is not proper,
a limit map may not exist
in $ X $, but it exists in
the \emph{ultrapower}
$ X^{\omega} $ 

We define very briefly
what is the ultrapower $ X^{\omega} $,
see \cite[Ch. 3]{alexanderAlexandrovGeometryPreliminary2019}.
A selective ultrafilter $ \omega $ on $ \N $
allows to define the $ \omega $-limit
(in $ \R \cup \left\{ -\infty, +\infty \right\} $)
of any sequence of real number $ (x_{n}) \in \R^{\N} $.
We denote this $ \omega $-limit by
$ \lim_{n \to \omega} x_{n} $;
intuively, it is the choice of a convergent subsequence.
We equip the product $ X^{\N} $ with the pseudo-metric
\begin{equation}
	\abs{(x_{n}) (y_{n})} = \lim_{n \to \omega} \abs{x_{n} y_{n}}
	,
\end{equation}
and the associated metric space
is denoted by $ X^{\omega} $.
The point in $ X^{\omega} $
associated to a sequence $ (x_{n}) $
is denoted by $ \lim_{n \to \omega} x_{n} $.
There is an isometric embedding
$ X \to X^{\omega} $
which maps $ x $ to
$ \lim_{n \to \omega} x $.
We consider $ X $ as a isometricaly
embedded subspace of $ X^{\omega} $
and any sequence $ (x_{n}) $ of $ X $
has an $ \omega $-limit
$ \lim_{n \to \omega} x_{n} $
in $ X^{\omega} $,
which is a limit of a convergent subsequence.
The ultraproduct of a $ \cat{\kappa} $
space is a $ \cat{\kappa} $ space.

Going back to the argument;
the sequence of geodesic polygons
$ P_{\varepsilon} $
has a $ \omega $-limit $ P $
in $ X^{\omega} $
when $ \varepsilon $ goes to 0.
We can proceed with the same argument,
but in $ X^{\omega} $ instead of $ X $,
and still reach
a contradiction.


\printbibliography
\end{document}